\documentclass[11pt,a4paper]{article}

\usepackage[T1]{fontenc}
\usepackage[utf8]{inputenc}
\usepackage{lmodern}

\usepackage{amsmath,amssymb,amsthm,mathtools}
\usepackage[a4paper,margin=2.7cm]{geometry}
\usepackage{enumitem}

\usepackage{hyperref}
\hypersetup{
    colorlinks=true,
    linkcolor=blue,
    citecolor=blue,
    urlcolor=blue
}

\newtheorem{theorem}{Theorem}[section]
\newtheorem{proposition}[theorem]{Proposition}
\newtheorem{lemma}[theorem]{Lemma}
\newtheorem{corollary}[theorem]{Corollary}

\theoremstyle{definition}
\newtheorem{definition}[theorem]{Definition}
\newtheorem{example}[theorem]{Example}

\theoremstyle{remark}
\newtheorem{remark}[theorem]{Remark}

\newcommand{\A}{R(T)}
\newcommand{\Htwo}{L^2(T)}
\newcommand{\Hone}{L^1(T)}
\newcommand{\BT}{\mathcal B_T}
\newcommand{\BTr}{\mathcal B_T^r}
\newcommand{\inner}[2]{\left\langle #1,#2\right\rangle_T}
\newcommand{\norm}[1]{\left\|#1\right\|_{T,2}}
\newcommand{\opnorm}[1]{\left\|#1\right\|_{T,\mathrm{op}}}

\title{Adjoints, Order and Spectral Moduli on $L^2(T)$}
\author{Mohamed Amine BEN AMOR \\\small{Research Laboratory of Algebra, Topology, Arithmetic, and Order }\\
 \small{Department of Mathematics, Faculty of Mathematical, Physical and Natural Sciences of Tunis} \\
\small{Tunis-El Manar University, 2092-El Manar, Tunisia}}
\date{}

\begin{document}

\maketitle

\begin{abstract}
We study adjoints of $R(T)$-linear, $T$-strongly bounded operators on the
conditional space $L^2(T)$ and the interaction between its Hilbert-type
geometry and its Riesz-space order.  This produces two natural operator
moduli: the Riesz--Kantorovich order modulus and the spectral modulus
$(S^*S)^{1/2}$.  After isolating the $T$-regular operators, we show that
they form a Dedekind-complete order ideal in the natural lattice of
order-bounded $R(T)$-module homomorphisms.  On this operator lattice the
order operations remain $T$-strongly bounded and the adjoint is an order
automorphism.  We then compare the two moduli and explain why they may differ.
\end{abstract}



\section{Introduction}

For a Hilbert space $\mathcal K$ and a bounded linear operator
$S\in\mathcal B(\mathcal K)$, the existence of the adjoint is one of the
most basic consequences of the classical Riesz representation theorem.
For each $y\in\mathcal K$, the map
\[
    x\longmapsto \langle Sx,y\rangle
\]
is a bounded linear functional; hence there is a unique vector $S^*y$ such
that
\[
    \langle Sx,y\rangle=\langle x,S^*y\rangle.
\]
This argument is standard in Hilbert-space theory; see, for example,
\cite{ConwayFA,ReedSimonI}.  The main point is not completeness alone, but
the representation of continuous linear functionals by inner products.

When scalar-valued inner products are replaced by algebra-valued inner
products, adjoints become subtler.  For Hilbert $C^*$-modules, a bounded
module map need not be adjointable, and adjointable maps therefore form a
distinguished operator class; see Lance \cite{Lance}.  In the theory of
Kaplansky--Hilbert modules, strong completeness and self-duality restore
many features familiar from Hilbert spaces, including a spectral theory
for normal operators; see Kaplansky \cite{KaplanskyHilbert} and Wright
\cite{WrightSpectral}.

The conditional $L^2$-spaces associated with conditional expectation
operators on Riesz spaces provide another natural algebra-valued setting.
They arise naturally in the measure-free probability setting of Riesz spaces.  Kuo, Labuschagne and Watson established the basic theory of
conditional expectation operators, including their averaging property and
extension to natural domains \cite{KuoLabuschagneWatson}.  Labuschagne and Watson introduced the $L^2$-type setting used for discrete
stochastic integration \cite{LabuschagneWatson}.  The notion of
$T$-strong convergence was developed by Azouzi, Kuo, Ramdane and Watson
\cite{AzouziKuoRamdaneWatson}, while Azouzi and Trabelsi used functional
calculus to develop the general $L^p(T)$ scale \cite{AzouziTrabelsi}.
Strong completeness results for the natural domain and related conditional
$L^p$ spaces were obtained by Kuo, Rodda and Watson
\cite{KuoRoddaWatson}.  The strong completeness and Riesz--Fr\'echet
representation results for $L^2(T)$ needed here are due to Kalauch, Kuo and
Watson \cite{KalauchKuoWatsonStrong,KalauchKuoWatsonRF}.

Let $T$ be a strictly positive conditional expectation operator and put
$A=R(T)$.  The space $L^2(T)$ carries the $A$-valued pairing
\[
    \langle x,y\rangle_T=T(xy)
\]
and the $A$-valued norm
\[
    \|x\|_{T,2}=\bigl(T(x^2)\bigr)^{1/2}.
\]
The Riesz--Fr\'echet representation theorem of Kalauch, Kuo and Watson
\cite{KalauchKuoWatsonRF} identifies $L^2(T)$ with its $T$-strong dual.
This suggests that the classical Hilbert-space construction of adjoints
should survive in the conditional setting.

The aim of this paper is to develop this observation and to use the two
structures carried by $L^2(T)$.
We first construct adjoints for $R(T)$-linear, $T$-strongly bounded
operators, prove the involution and the $R(T)$-valued operator-norm
isometry, and distinguish order positivity from inner-product positivity.
The main point is then the comparison of two operator moduli.  The
Riesz--Kantorovich modulus belongs to the lattice order, whereas
$(S^*S)^{1/2}$ belongs to the Hilbert-type geometry.  For this comparison, we use the $T$-regular operators, for which the
Riesz--Kantorovich parts remain $T$-strongly bounded.  On this
operator lattice the adjoint preserves lattice operations, and the two
moduli can be compared without applying the adjoint outside its domain.
The resulting cross-term formula explains why the two constructions agree
in some cases, such as multiplication operators, but may differ already in
dimension two.

% ============================================================
\section{Preliminaries}
% ============================================================

\subsection*{Standing assumptions}
Throughout the paper, $E$ denotes a Dedekind complete Riesz space with
weak order unit $e$, and
\[
    T:E\longrightarrow E
\]
denotes a strictly positive conditional expectation operator such that
\[
    Te=e.
\]
Thus $T$ is, throughout, a positive order-continuous projection.  These
hypotheses are standing assumptions and will not be repeated in individual
statements.

We work in the natural domain $\Hone$ of $T$. As usual, after passing
to the natural domain, we continue to denote the extended conditional
expectation operator by $T$; the same standing assumptions and the same
weak order unit $e$ are retained.

Set
\[
    A:=\A.
\]

In the standard definition used in the Riesz-space probability literature,
a conditional expectation has Dedekind complete Riesz-subspace range; see
\cite{KuoLabuschagneWatson}.  Thus $A=R(T)$ is Dedekind complete in the
inherited lattice order.  The role of strict positivity should be kept
separate from this definition.  For instance, if one starts only from a
strictly positive positive projection, then its range is automatically a
Riesz subspace: if $y=Ty$, positivity gives $|y|\le T|y|$, while
\[
 d:=T|y|-|y|\ge0,
 \qquad Td=0;
\]
strict positivity forces $d=0$, and hence $|y|=T|y|\in R(T)$.

The multiplication used below is the one inherited from the universal
completion and the natural-domain construction.  The averaging theorem for
conditional expectations gives
\[
 T(az)=aTz,\qquad a\in R(T),
\]
whenever the products are defined, and in this framework $A=R(T)$ is a
unital Archimedean $f$-algebra with unit $e$, while $L^1(T)$ and $L^2(T)$
are $A$-modules; see \cite{KuoLabuschagneWatson,LabuschagneWatson}.

Since $A$ is Dedekind complete, it is uniformly complete.  We shall use two
standard consequences for uniformly complete unital $f$-algebras.  First,
the bounded inversion property implies that if
$a\ge\varepsilon e$ for some real $\varepsilon>0$, then $a$ is invertible,
$a^{-1}\ge0$, and
\[
 0\le a^{-1}\le \varepsilon^{-1}e.
\]
Second, every positive element of $A$ has a unique positive square root.
These facts are classical; more precisely, the bounded inversion property
and existence of positive square roots are given in
\cite[Theorems~3.4 and~3.9]{HuijsmansDePagterIdeal}; see also
\cite[Theorem~4.2 and Corollary~4.3]{BeukersHuijsmansDePagter} for the
square-root result.  They justify every scalar inverse and scalar square
root used below.

Although $R(T)$ may enjoy stronger completeness properties in the
natural-domain setting, the arguments in this paper use only its Dedekind
completeness, its unital $f$-algebra structure, and the $T$-strong
completeness of $L^2(T)$.  In particular, universal completeness is not an
additional hypothesis in any of the arguments below.

Following the conditional $L^2$ notation used in this literature, define
\[
    H:=\Htwo
    :=
    \left\{
        x\in\Hone : x^2\in\Hone
    \right\}.
\]
Then $H$ is an $A$-module.

\begin{lemma}
\label{lem:L2-Dedekind-complete}
The Riesz space $H=L^2(T)$ is Dedekind complete.
\end{lemma}

\begin{proof}
This fact is already contained in the early $L^2(T)$ construction of
Labuschagne and Watson \cite[Lemma~2.2]{LabuschagneWatson}.  We include the
short argument because Dedekind completeness is used essentially below in
the Riesz--Kantorovich calculus.  In the notation of that paper,
$L^1(T)=\operatorname{dom}(T)$ is an order-dense order ideal of the
universal completion $E^u$, and
\[
 L^2(T)=\operatorname{dom}_2(T)
 =\{x\in\operatorname{dom}(T):x^2\in\operatorname{dom}(T)\}.
\]
First $L^2(T)$ is an order ideal of $L^1(T)$.  Indeed, if
$y\in L^2(T)$ and $x\in L^1(T)$ satisfy $|x|\le |y|$, then, in the
$f$-algebra $E^u$,
\[
 0\le x^2=|x|^2\le |y|^2=y^2\in L^1(T).
\]
Since $L^1(T)$ is an order ideal of $E^u$, it follows that
$x^2\in L^1(T)$, and hence $x\in L^2(T)$.

Now let $B\subset L^2(T)$ be nonempty and bounded above in $L^2(T)$, say
$b\le u$ for all $b\in B$ with $u\in L^2(T)$.  Since $L^1(T)$ is
Dedekind complete, $s:=\sup_{L^1(T)}B$ exists.  Fix $b_0\in B$.  Then
\[
 0\le s-b_0\le u-b_0.
\]
Because $u-b_0\in L^2(T)$ and $L^2(T)$ is an order ideal of $L^1(T)$,
we have $s-b_0\in L^2(T)$, hence $s\in L^2(T)$.  Thus $s$ is also the
supremum of $B$ in $L^2(T)$.
\end{proof}

\begin{remark}
Labuschagne and Watson state that $\operatorname{dom}_2(T)$ is an order
ideal of $\operatorname{dom}(T)$ and is therefore Dedekind complete
\cite[Lemma~2.2]{LabuschagneWatson}.  We include the proof only because this
property is used later for regular operators and their Riesz--Kantorovich
moduli.
\end{remark}

For $x,y\in H$, define
\[
    \inner{x}{y}
    :=
    T(xy).
\]
Since $x,y\in L^2(T)$, the product $xy$ belongs to $L^1(T)$ by the
conditional Cauchy--Schwarz inequality, so the preceding expression
is well-defined.

The associated $A$-valued norm is
\[
    \norm{x}
    :=
    \bigl(T(x^2)\bigr)^{1/2}.
\]

The conditional Cauchy--Schwarz inequality reads
\begin{equation}
\label{eq:conditional-CS}
    |T(xy)|
    \leq
    \norm{x}\norm{y},
    \qquad x,y\in H.
\end{equation}
The use of the $R(T)$-valued norm and of $T$-strong convergence is
consistent with the convergence theory developed by Azouzi, Kuo, Ramdane
and Watson \cite{AzouziKuoRamdaneWatson}.  Strong completeness of the
$L^2(T)$ space in precisely this conditional norm is supplied by
Kalauch, Kuo and Watson \cite{KalauchKuoWatsonStrong}; this completeness
will be essential when polynomial approximants are used below to construct
operator square roots.

We shall also repeatedly use the averaging property
\begin{equation}
\label{eq:averaging}
    T(az)=aT(z),
    \qquad a\in A,
\end{equation}
whenever the products involved are defined.

% ============================================================
\section{The $T$-strong dual and the Riesz--Fr\'echet theorem}
% ============================================================

We recall the notion of $T$-strong boundedness used by
Kalauch, Kuo and Watson.

\begin{definition}
An additive $A$-homogeneous functional
\[
    \varphi:H\longrightarrow A
\]
is called \emph{$T$-strongly bounded} if there exists $k\in A_+$
such that
\begin{equation}
\label{eq:strong-functional}
    |\varphi(x)|
    \leq
    k\norm{x},
    \qquad x\in H.
\end{equation}

The collection of all such functionals is denoted by
\[
    \widehat H
\]
and is called the $T$-strong dual of $H$.
\end{definition}

For later use we also record explicitly the $A$-valued norm on the
$T$-strong dual.  For $\varphi\in\widehat H$ set
\[
 \|\varphi\|_{T,*}
 :=\inf\Bigl\{k\in A_+:
       |\varphi(x)|\le k\norm{x}\ \text{for every }x\in H\Bigr\}.
\]
In the Riesz--Fr\'echet theory of Kalauch--Kuo--Watson this infimum is an
admissible bound and defines an $A$-valued norm on $\widehat H$; under the
representation map below one has
\[
 \|\Psi(z)\|_{T,*}=\norm{z}.
\]

The key ingredient in what follows is the Riesz--Fr\'echet
representation theorem established by Kalauch, Kuo and Watson
\cite{KalauchKuoWatsonRF,KalauchKuoWatsonStrong}.

\begin{theorem}
\label{thm:RF}
The map
\[
    \Psi:H\longrightarrow\widehat H,
    \qquad
    \Psi(z)(x):=T(xz),
\]
is an additive, $A$-homogeneous, norm-preserving bijection.

Consequently, for every $\varphi\in\widehat H$ there exists a unique
$z\in H$ such that
\begin{equation}
\label{eq:RFrepresentation}
    \varphi(x)
    =
    T(xz)
    =
    \inner{x}{z},
    \qquad x\in H.
\end{equation}
\end{theorem}

This representation theorem plays the role of self-duality for the
$A$-valued inner product space $H$.

% ============================================================
\section{Adjointability forces $R(T)$-linearity}
% ============================================================

First, we note that the adjoint identity itself implies module linearity.

\begin{proposition}\label{prop:automatic-linearity}
Let $S:H\to H$ and $R:H\to H$ be arbitrary maps such that
\begin{equation}\label{eq:formal-adjoint}
    \inner{Sx}{y}=\inner{x}{Ry},\qquad x,y\in H.
\end{equation}
Then $S$ and $R$ are automatically additive and $A$-homogeneous.  In
particular, they are $A$-linear.
\end{proposition}

\begin{proof}
Let $x_1,x_2,y\in H$.  From \eqref{eq:formal-adjoint},
\[
 \inner{S(x_1+x_2)-Sx_1-Sx_2}{y}=0.
\]
Taking
\[
 y=S(x_1+x_2)-Sx_1-Sx_2
\]
gives
\[
 T\!\left((S(x_1+x_2)-Sx_1-Sx_2)^2\right)=0.
\]
Strict positivity of $T$ implies
$S(x_1+x_2)=Sx_1+Sx_2$.

Now let $a\in A$.  Using the averaging property,
\[
\begin{aligned}
 \inner{S(ax)}{y}
 &=\inner{ax}{Ry}
 =a\inner{x}{Ry}
 =a\inner{Sx}{y}
 =\inner{aSx}{y}.
\end{aligned}
\]
Taking $y=S(ax)-aSx$ yields
\[
 T((S(ax)-aSx)^2)=0,
\]
and strict positivity gives $S(ax)=aSx$.  Hence $S$ is $A$-linear.
Interchanging the two variables in \eqref{eq:formal-adjoint}, using the
symmetry of the real pairing, gives the same conclusion for $R$.
\end{proof}

\begin{remark}
Thus $A$-linearity is not an artificial feature of the definition of an
adjoint.  Any pair of maps satisfying the adjoint identity is necessarily
$A$-linear.  The $A$-linearity assumption in the existence theorem below
is instead the hypothesis that ensures, for each fixed $y$, that
$x\mapsto\inner{Sx}{y}$ belongs to the $T$-strong dual to which the
Riesz--Fr\'echet theorem applies.
\end{remark}

% ============================================================
\section{$T$-strongly bounded operators}
% ============================================================

We now introduce the operator class that will be used throughout.

\begin{definition}
An $A$-linear operator
\[
    S:H\longrightarrow H
\]
is called \emph{$T$-strongly bounded} if there exists $k\in A_+$
such that
\begin{equation}
\label{eq:operator-bound}
    \norm{Sx}
    \leq
    k\norm{x},
    \qquad x\in H.
\end{equation}

We denote the collection of all such operators by
\[
    \boxed{
    \BT(H).
    }
\]
\end{definition}

It is immediate that $\BT(H)$ is closed under addition and
composition. Indeed, if
\[
    \norm{Sx}\leq k\norm{x},
    \qquad
    \norm{Rx}\leq \ell\norm{x},
\]
then
\[
    \norm{(S+R)x}
    \leq
    (k+\ell)\norm{x},
\]
and
\[
    \norm{SRx}
    \leq
    k\norm{Rx}
    \leq
    k\ell\norm{x}.
\]

Moreover, the identity operator belongs to $\BT(H)$.

Thus $\BT(H)$ is a unital algebra under composition.

% ============================================================
\section{Existence and uniqueness of adjoints}
% ============================================================

We now derive the existence of adjoints from the
Riesz--Fr\'echet representation theorem.

\begin{theorem}
\label{thm:adjoint}
Let
\[
    S\in\BT(H).
\]
Then there exists a unique operator
\[
    S^*\in\BT(H)
\]
such that
\begin{equation}
\label{eq:adjoint}
    \boxed{
    \inner{Sx}{y}
    =
    \inner{x}{S^*y},
    \qquad x,y\in H.
    }
\end{equation}

If $k\in A_+$ satisfies
\[
    \norm{Sx}\leq k\norm{x}
    \qquad(x\in H),
\]
then the same $k$ satisfies
\begin{equation}
\label{eq:adjoint-same-bound}
    \norm{S^*y}
    \leq
    k\norm{y},
    \qquad y\in H.
\end{equation}
\end{theorem}

\begin{proof}
Fix $y\in H$ and define
\[
    \varphi_y:H\longrightarrow A
\]
by
\[
    \varphi_y(x)
    :=
    \inner{Sx}{y}
    =
    T((Sx)y).
\]

We first show that $\varphi_y\in\widehat H$.

Since $S$ is additive, $\varphi_y$ is additive.

Let $a\in A$. Since $S$ is $A$-linear,
\[
    S(ax)=aSx.
\]
Hence, using the averaging property \eqref{eq:averaging},
\[
\begin{aligned}
    \varphi_y(ax)
    &=
    T(S(ax)y)\\
    &=
    T(a(Sx)y)\\
    &=
    aT((Sx)y)\\
    &=
    a\varphi_y(x).
\end{aligned}
\]
Thus $\varphi_y$ is $A$-homogeneous.

By the conditional Cauchy--Schwarz inequality
\eqref{eq:conditional-CS},
\[
\begin{aligned}
    |\varphi_y(x)|
    &=
    |T((Sx)y)|\\
    &\leq
    \norm{Sx}\norm{y}.
\end{aligned}
\]
If
\[
    \norm{Sx}\leq k\norm{x},
\]
then
\[
    |\varphi_y(x)|
    \leq
    k\norm{y}\norm{x}.
\]
Since
\[
    k\norm{y}\in A_+,
\]
it follows that
\[
    \varphi_y\in\widehat H.
\]

By Theorem~\ref{thm:RF}, there exists a unique element
$z_y\in H$ such that
\[
    \varphi_y(x)=T(xz_y),
    \qquad x\in H.
\]

Define
\[
    S^*y:=z_y.
\]
Then
\[
    \inner{Sx}{y}
    =
    \inner{x}{S^*y}
\]
for every $x,y\in H$.

We now prove that $S^*$ is $A$-linear.

Let $y_1,y_2\in H$. For every $x\in H$,
\[
\begin{aligned}
    \inner{x}{S^*(y_1+y_2)}
    &=
    \inner{Sx}{y_1+y_2}\\
    &=
    \inner{Sx}{y_1}
    +
    \inner{Sx}{y_2}\\
    &=
    \inner{x}{S^*y_1}
    +
    \inner{x}{S^*y_2}\\
    &=
    \inner{x}{S^*y_1+S^*y_2}.
\end{aligned}
\]
The uniqueness part of the Riesz--Fr\'echet theorem yields
\[
    S^*(y_1+y_2)
    =
    S^*y_1+S^*y_2.
\]

Similarly, if $a\in A$, then
\[
\begin{aligned}
    \inner{x}{S^*(ay)}
    &=
    \inner{Sx}{ay}\\
    &=
    a\inner{Sx}{y}\\
    &=
    a\inner{x}{S^*y}\\
    &=
    \inner{x}{aS^*y}.
\end{aligned}
\]
Again by uniqueness,
\[
    S^*(ay)=aS^*y.
\]
Thus $S^*$ is $A$-linear.

It remains to prove that $S^*$ is $T$-strongly bounded.

By the norm-preserving part of Theorem~\ref{thm:RF}, the representing
element $S^*y$ has the same $A$-valued norm as the functional
$\varphi_y$. Since
\[
    |\varphi_y(x)|
    \leq
    k\norm{y}\norm{x},
\]
the definition of the $T$-strong dual norm yields
\[
    \norm{S^*y}
    \leq
    k\norm{y}.
\]
Hence $S^*\in\BT(H)$.

Finally, suppose that $R:H\to H$ also satisfies
\[
    \inner{Sx}{y}
    =
    \inner{x}{Ry}
\]
for every $x,y\in H$. Then
\[
    \inner{x}{S^*y}
    =
    \inner{x}{Ry}
\]
for every $x\in H$. Thus $S^*y$ and $Ry$ represent the same
functional on $H$. By uniqueness in Theorem~\ref{thm:RF},
\[
    S^*y=Ry.
\]
Since $y$ is arbitrary,
\[
    S^*=R.
\]
\end{proof}

% ============================================================
\section{The involution on $\mathcal B_T(H)$}
% ============================================================

We now establish the elementary algebraic properties of the adjoint.

\begin{proposition}
\label{prop:adjoint-properties}
Let $S,R\in\BT(H)$ and $a\in A$. Then
\begin{enumerate}[label=\textup{(\roman*)}]
    \item
    \[
        (S+R)^*=S^*+R^*;
    \]

    \item
    \[
        (aS)^*=aS^*;
    \]

    \item
    \[
        (SR)^*=R^*S^*;
    \]

    \item
    \[
        (S^*)^*=S;
    \]

    \item
    \[
        I^*=I.
    \]
\end{enumerate}
\end{proposition}

\begin{proof}
For $x,y\in H$,
\[
\begin{aligned}
    \inner{(S+R)x}{y}
    &=
    \inner{Sx}{y}
    +
    \inner{Rx}{y}\\
    &=
    \inner{x}{S^*y+R^*y}.
\end{aligned}
\]
Uniqueness of the adjoint gives
\[
    (S+R)^*=S^*+R^*.
\]

Similarly,
\[
\begin{aligned}
    \inner{(aS)x}{y}
    &=
    a\inner{Sx}{y}\\
    &=
    a\inner{x}{S^*y}\\
    &=
    \inner{x}{aS^*y},
\end{aligned}
\]
so
\[
    (aS)^*=aS^*.
\]

For composition,
\[
\begin{aligned}
    \inner{SRx}{y}
    &=
    \inner{Rx}{S^*y}\\
    &=
    \inner{x}{R^*S^*y}.
\end{aligned}
\]
Hence
\[
    (SR)^*=R^*S^*.
\]

To prove the involution identity, note that the inner product is
symmetric in the real setting:
\[
    \inner{x}{y}=\inner{y}{x}.
\]
Thus
\[
\begin{aligned}
    \inner{S^*x}{y}
    &=
    \inner{y}{S^*x}\\
    &=
    \inner{Sy}{x}\\
    &=
    \inner{x}{Sy}.
\end{aligned}
\]
Therefore $S$ is an adjoint of $S^*$. By uniqueness,
\[
    (S^*)^*=S.
\]

Finally,
\[
    \inner{Ix}{y}
    =
    \inner{x}{Iy},
\]
hence
\[
    I^*=I.
\]
\end{proof}

\begin{corollary}
The map
\[
    *:\BT(H)\longrightarrow\BT(H),
    \qquad
    S\longmapsto S^*,
\]
is an involutive anti-automorphism of the unital algebra
$\BT(H)$.
\end{corollary}

% ============================================================
\section{The $R(T)$-valued operator norm}
% ============================================================

We next introduce the natural operator norm.

For $S\in\BT(H)$, define the set of admissible bounds
\[
    \mathcal M(S)
    :=
    \left\{
        k\in A_+:
        \norm{Sx}\leq k\norm{x}
        \text{ for all }x\in H
    \right\}.
\]
By definition of $\BT(H)$,
\[
    \mathcal M(S)\neq\varnothing.
\]

Since $A=R(T)$ is Dedekind complete, the infimum
\[
    \inf\mathcal M(S)
\]
exists in $A_+$.

\begin{definition}
For $S\in\BT(H)$, define
\[
    \boxed{
    \opnorm{S}
    :=
    \inf\mathcal M(S).
    }
\]
\end{definition}

The analogous infimum construction for the $T$-strong dual of
$L^2(T)$ is used by Kalauch--Kuo--Watson, where it is proved to give an
$R(T)$-valued norm and, importantly, the infimum is itself an admissible
bound.  We record the operator version, since this fact will be used
repeatedly below.

\begin{proposition}
\label{prop:operator-norm-is-norm}
For every $S\in\BT(H)$ one has
\[
    \norm{Sx}\leq \opnorm{S}\,\norm{x},
    \qquad x\in H.
\]
In particular, $\opnorm{S}\in\mathcal M(S)$.  Moreover,
$S\mapsto\opnorm{S}$ is an $A$-valued norm on $\BT(H)$; that is,
for $R,S\in\BT(H)$ and $a\in A$,
\[
\begin{aligned}
    &\opnorm{S}\geq0,\qquad
      \opnorm{S}=0\Longleftrightarrow S=0,\\
    &\opnorm{aS}=|a|\,\opnorm{S},\\
    &\opnorm{R+S}\leq\opnorm{R}+\opnorm{S}.
\end{aligned}
\]
It is also submultiplicative:
\[
    \opnorm{RS}\leq\opnorm{R}\,\opnorm{S}.
\]
\end{proposition}

\begin{proof}
Put $m=\inf\mathcal M(S)$.  Fix $x\in H$ and write
$b=\norm{x}\in A_+$.  For every $k\in\mathcal M(S)$,
\[
    \norm{Sx}\leq kb.
\]
Multiplication by the fixed positive element $b$ is order continuous in
$A$; hence it preserves infima of nonempty lower-bounded sets.  Therefore
\[
    \norm{Sx}
    \leq \inf_{k\in\mathcal M(S)}kb
    =\left(\inf_{k\in\mathcal M(S)}k\right)b
    =m\norm{x}.
\]
Thus $m\in\mathcal M(S)$, proving the asserted optimal-bound property.

Positivity is immediate.  If $\opnorm{S}=0$, the optimal-bound property
gives $\norm{Sx}=0$ for every $x\in H$.  Since $\norm{\cdot}$ is an
$A$-valued norm on $H$, $Sx=0$ for every $x$, and hence $S=0$.  The
converse is clear.

For $x\in H$,
\[
\begin{aligned}
    \norm{(R+S)x}
    &\leq \norm{Rx}+\norm{Sx}\\
    &\leq (\opnorm{R}+\opnorm{S})\norm{x},
\end{aligned}
\]
so the optimality of the infimum gives
\[
    \opnorm{R+S}\leq\opnorm{R}+\opnorm{S}.
\]

Let $a\in A$.  Since
\[
    \norm{aSx}=|a|\,\norm{Sx},
\]
we have $\mathcal M(aS)=\mathcal M(|a|S)$, and hence it is enough to
consider $a\geq0$.  The optimal-bound property gives immediately
\[
    \opnorm{aS}\leq a\,\opnorm{S}.
\]
To obtain the reverse inequality without assuming universal completeness,
fix $\varepsilon>0$ and put
$c=a+\varepsilon e$.  Since $c\geq\varepsilon e$, the element $c$ is
invertible in the uniformly complete unital Archimedean $f$-algebra $A$,
and $c^{-1}\geq0$.  Applying the preceding inequality first to $cS$ and
then to $S=c^{-1}(cS)$ gives
\[
    \opnorm{cS}=c\,\opnorm{S}.
\]
On the other hand, by the triangle inequality established above and the
already established inequality for scalar multiplication,
\[
    c\,\opnorm{S}
    =\opnorm{(a+\varepsilon e)S}
    \leq \opnorm{aS}+\varepsilon\opnorm{S}.
\]
Cancelling the common term $\varepsilon\opnorm{S}$ in the ordered vector
space $A$ yields
\[
    a\,\opnorm{S}\leq\opnorm{aS}.
\]
Thus $\opnorm{aS}=a\opnorm{S}$ for $a\geq0$, and therefore
$\opnorm{aS}=|a|\opnorm{S}$ for arbitrary $a\in A$.

Finally,
\[
    \norm{RSx}
    \leq \opnorm{R}\norm{Sx}
    \leq \opnorm{R}\opnorm{S}\norm{x},
\]
and hence $\opnorm{RS}\leq\opnorm{R}\opnorm{S}$.
\end{proof}

\begin{remark}
The preceding proposition is why we may legitimately call
$\opnorm{\cdot}$ the $R(T)$-valued operator norm.  The point is not merely
terminological: the defining infimum is an admissible bound.  Compare the
corresponding construction for the $T$-strong dual in
Kalauch--Kuo--Watson, where the same phenomenon is established for
$R(T)$-valued bounded functionals.
\end{remark}

The existence theorem for adjoints also shows directly that the adjoint
preserves the operator norm.

\begin{theorem}
\label{thm:adjoint-norm}
For every $S\in\BT(H)$,
\[
    \boxed{
    \opnorm{S^*}
    =
    \opnorm{S}.
    }
\]
\end{theorem}

\begin{proof}
Let $k\in\mathcal M(S)$.  Thus
\[
    \norm{Sx}\leq k\norm{x},\qquad x\in H.
\]
By Theorem~\ref{thm:adjoint}, the same $k$ satisfies
\[
    \norm{S^*y}\leq k\norm{y},\qquad y\in H.
\]
Hence $k\in\mathcal M(S^*)$, and therefore
\[
    \opnorm{S^*}\leq\opnorm{S}.
\]
Applying the same argument to $S^*$ and using $(S^*)^*=S$ gives
\[
    \opnorm{S}\leq\opnorm{S^*}.
\]
The two inequalities yield the result.
\end{proof}

% ============================================================
\section{Order structure and $T$-regular operators}
% ============================================================

The Riesz-space order on $H=L^2(T)$ provides notions that have no direct
counterpart in an abstract Hilbert space.  We first distinguish order
positivity from positivity defined by the $A$-valued inner product and then
identify the operator class on which the Riesz--Kantorovich lattice
operations are compatible with $T$-strong boundedness.

\begin{definition}
An operator $S\in\BT(H)$ is called \emph{order positive}, written
$S\geq_o0$, if
\[
    x\geq0\quad\Longrightarrow\quad Sx\geq0.
\]
It is called \emph{inner-product positive}, written $S\geq_h0$, if
\[
    \inner{Sx}{x}\geq0,\qquad x\in H.
\]
It is \emph{self-adjoint} if $S^*=S$.
\end{definition}

The following separation observation will be used repeatedly.

\begin{lemma}\label{lem:positive-separation}
For $z\in H$, the following are equivalent:
\begin{enumerate}[label=\textup{(\roman*)}]
\item $z\geq0$;
\item $T(xz)\geq0$ for every $x\in H_+$.
\end{enumerate}
\end{lemma}

\begin{proof}
The implication (i)$\Rightarrow$(ii) follows from positivity of $T$.
Conversely, suppose that $T(xz)\geq0$ for every $x\geq0$.  Taking
$x=z^-$ gives
\[
 z^-z=-\,(z^-)^2,
\]
because $z^+\perp z^-$.  Hence
\[
 0\leq T(z^-z)=-T((z^-)^2)\leq0.
\]
Thus $T((z^-)^2)=0$.  Strict positivity of $T$ gives $z^-=0$, and
therefore $z\geq0$.
\end{proof}

\begin{theorem}
\label{thm:order-positive-adjoint}
For every $S\in\BT(H)$,
\[
    \boxed{S\geq_o0\quad\Longleftrightarrow\quad S^*\geq_o0.}
\]
\end{theorem}

\begin{proof}
Assume $S\geq_o0$ and let $y\in H_+$.  For every $x\in H_+$,
\[
 T(xS^*y)=\inner{x}{S^*y}=\inner{Sx}{y}=T((Sx)y)\geq0.
\]
Lemma~\ref{lem:positive-separation} gives $S^*y\geq0$.  Hence
$S^*\geq_o0$.  Applying the same implication to $S^*$ and using
$(S^*)^*=S$ gives the converse.
\end{proof}

\begin{remark}
There is no immediate analogue of Theorem~\ref{thm:order-positive-adjoint}
for the naive condition $x>0\Rightarrow Sx>0$.  Even when $Sx>0$ and
$y>0$, the product $(Sx)y$ may vanish if the two elements are disjoint.
Any transfer theorem for a stronger positivity notion would therefore
require additional support or irreducibility hypotheses.
\end{remark}

\begin{proposition}\label{prop:SstarS-positive}
For every $S\in\BT(H)$, the operator $S^*S$ is self-adjoint and
inner-product positive.  More precisely,
\[
    \inner{S^*Sx}{x}=\norm{Sx}^{2}\geq0,
    \qquad x\in H.
\]
Moreover,
\[
    \ker(S^*S)=\ker S.
\]
\end{proposition}

\begin{proof}
Using Proposition~\ref{prop:adjoint-properties},
\[
    (S^*S)^*=S^*(S^*)^*=S^*S.
\]
Furthermore,
\[
    \inner{S^*Sx}{x}=\inner{Sx}{Sx}
    =T((Sx)^2)=\norm{Sx}^{2}\geq0.
\]
If $Sx=0$, then clearly $S^*Sx=0$.  Conversely, if $S^*Sx=0$, then
$T((Sx)^2)=0$, and strict positivity of $T$ yields $Sx=0$.
\end{proof}

\begin{remark}
Inner-product positivity of $S^*S$ is not the same as order positivity.
The inequality $\inner{S^*Sx}{x}\geq0$ for all $x$ does not imply
$x\geq0\Rightarrow S^*Sx\geq0$.
\end{remark}

\subsection{Riesz--Kantorovich operations and module linearity}

Since $H$ is Dedekind complete, the classical Riesz--Kantorovich theorem
makes the real vector space $L_b(H,H)$ of order-bounded real-linear
operators into a Dedekind complete vector lattice.  For $R\in L_b(H,H)$
and $x\in H_+$,
\begin{equation}
\label{eq:RK-modulus}
 |R|_{\rm ord}x=\sup\{\,|Ru|:|u|\le x\,\}.
\end{equation}
See \cite[Theorem~1.16]{AbramovichAliprantis} and
\cite[Chapter~1]{ABPositiveOperators}.  At the module level, write
\[
 \mathcal L_b^A(H):=
 \{R:H\to H:\ R\text{ is $A$-linear and order bounded}\}.
\]
The Riesz--Kantorovich theorem of Chamberlain and Wortel applies to the
present setting: $H$ is an $A$-vector lattice, hence is directed and has
the Riesz decomposition property; it is Dedekind complete by
Lemma~\ref{lem:L2-Dedekind-complete}; and its $A$-valued norm places it in
the normed-module class for which the required sequential
$\mathbb P$-Archimedean condition holds.  Therefore
$\mathcal L_b^A(H)$ is a Dedekind-complete vector lattice, its lattice
operations are the Riesz--Kantorovich operations, and increasing suprema are
computed pointwise on $H_+$; see
\cite[Theorem~5.1]{ChamberlainWortel}.  In particular, lattice operations
of order-bounded $A$-module homomorphisms remain $A$-module
homomorphisms.  For later estimates we record a direct proof of the
module-linearity of the modulus.

\begin{lemma}
\label{lem:RK-module-linearity}
Let $R:H\to H$ be $A$-linear and order bounded as a real-linear operator.
Then $|R|_{\rm ord}$, $R^+$ and $R^-$ are $A$-linear.
\end{lemma}

\begin{proof}
It is enough to prove $A$-homogeneity of $|R|_{\rm ord}$.  Fix
$a\in A_+$ and $x\in H_+$.  For every $u$ with $|u|\le x$ one has
$|au|\le ax$ and
\[
 |R(au)|=|aRu|=a|Ru|.
\]
Multiplication by $a\ge0$ is an order-continuous lattice homomorphism on
$H$, hence taking suprema in \eqref{eq:RK-modulus} yields
\begin{equation}
\label{eq:RK-module-lower}
 a|R|_{\rm ord}x\le |R|_{\rm ord}(ax).
\end{equation}

Now fix a real $\varepsilon>0$ and put $a_\varepsilon=a+\varepsilon e$.
The bounded inversion property gives an inverse
$a_\varepsilon^{-1}\in A_+$.  The map
$u\mapsto a_\varepsilon u$ is therefore an order isomorphism from
$[-x,x]$ onto $[-a_\varepsilon x,a_\varepsilon x]$.  By $A$-linearity of
$R$ and \eqref{eq:RK-modulus},
\[
 |R|_{\rm ord}(a_\varepsilon x)
 =a_\varepsilon |R|_{\rm ord}x.
\]
Since $ax\le a_\varepsilon x$ and $|R|_{\rm ord}$ is positive,
\[
 |R|_{\rm ord}(ax)
 \le (a+\varepsilon e)|R|_{\rm ord}x.
\]
Taking $\varepsilon=1/n$ and using the Archimedean property of $H$ gives
\[
 |R|_{\rm ord}(ax)\le a|R|_{\rm ord}x.
\]
Together with \eqref{eq:RK-module-lower}, this proves equality for
$a\ge0$ and $x\ge0$.  Real linearity of $|R|_{\rm ord}$ then extends the
identity first to arbitrary $x\in H$ and then, by writing
$a=a^+-a^-$, to arbitrary $a\in A$.

Finally,
\[
 R^+=\frac{|R|_{\rm ord}+R}{2},\qquad
 R^-=\frac{|R|_{\rm ord}-R}{2},
\]
so $R^+$ and $R^-$ are $A$-linear as well.
\end{proof}

The preceding lemma settles module linearity, but not by itself
$T$-strong boundedness of the Riesz--Kantorovich parts.  This distinction
is essential.  We therefore isolate the class on which the latter property
is automatic.

\begin{definition}
Let
\[
 \boxed{
 \BTr(H):=\BT(H)_+-\BT(H)_+,
 }
\]
where $\BT(H)_+$ denotes the order-positive operators in $\BT(H)$.  An
element of $\BTr(H)$ will be called \emph{$T$-regular}.
\end{definition}

Every $T$-regular operator is classically regular, hence order bounded, but
the converse is not assumed.

\begin{theorem}
\label{thm:T-regular-lattice}
The space $\BTr(H)$ is a vector lattice for the order inherited from
$\mathcal L_b^A(H)$ and a subalgebra of $\BT(H)$.  More precisely, if
\[
 S=P-N,\qquad P,N\in\BT(H)_+,
\]
then $|S|_{\rm ord}\in\BT(H)_+$ and
\begin{equation}
\label{eq:modulus-domination}
 0\le |S|_{\rm ord}\le P+N.
\end{equation}
Consequently $S^+,S^-\in\BT(H)_+$.

If $k_P,k_N\in A_+$ are admissible bounds for $P$ and $N$, respectively,
then
\begin{equation}
\label{eq:modulus-strong-bound}
 \norm{|S|_{\rm ord}x}
 \le (k_P+k_N)\norm{x},\qquad x\in H.
\end{equation}
\end{theorem}

\begin{proof}
Because $P,N$ are positive,
\[
 -(P+N)\le P-N\le P+N.
\]
In the vector lattice $\mathcal L_b^A(H)$ this is equivalent to
\eqref{eq:modulus-domination}.  Lemma~\ref{lem:RK-module-linearity} shows
that $|S|_{\rm ord}$ is $A$-linear.

Put $U=P+N$.  For arbitrary $x\in H$, positivity of
$|S|_{\rm ord}$ and \eqref{eq:modulus-domination} give
\[
 \bigl||S|_{\rm ord}x\bigr|
 \le |S|_{\rm ord}|x|
 \le U|x|.
\]
The conditional norm is monotone: if $|u|\le v$ with $v\ge0$, then
$u^2\le v^2$, hence
$\norm{u}\le\norm{v}$ by positivity of $T$.  Therefore
\[
 \norm{|S|_{\rm ord}x}
 \le \norm{U|x|}
 \le (k_P+k_N)\norm{|x|}
 =(k_P+k_N)\norm{x}.
\]
Thus $|S|_{\rm ord}\in\BT(H)_+$.  The formulas
\[
 S^+=\frac{|S|_{\rm ord}+S}{2},\qquad
 S^-=\frac{|S|_{\rm ord}-S}{2}
\]
then show that $S^+,S^-\in\BT(H)_+$.  Hence $\BTr(H)$ is closed under
modulus and therefore under all lattice operations.  Finally, if
$S=P-N$ and $R=Q-M$ with $P,N,Q,M\in\BT(H)_+$, then
\[
 SR=(PQ+NM)-(PM+NQ),
\]
and all four compositions are order positive and $T$-strongly bounded.
Thus $SR\in\BTr(H)$.
\end{proof}

\subsection{Intrinsic order structure of \texorpdfstring{$\BTr(H)$}{BTr(H)}}

The class $\BTr(H)$ is not merely a device for ensuring that the order
modulus lies in the domain of the adjoint.  It is a natural solid operator
lattice inside the Dedekind-complete module-operator lattice
$\mathcal L_b^A(H)$.

\begin{proposition}
\label{prop:BTr-order-ideal}
The positive cone of $\BTr(H)$ is exactly the cone of order-positive
$T$-strongly bounded operators:
\begin{equation}
\label{eq:BTr-positive-cone}
 \boxed{\bigl(\BTr(H)\bigr)_+=\BT(H)_+.}
\end{equation}
Moreover, $\BTr(H)$ is an order ideal of $\mathcal L_b^A(H)$.  Equivalently,
if $R\in\mathcal L_b^A(H)$ and $S\in\BTr(H)$ satisfy
\[
 |R|_{\rm ord}\le |S|_{\rm ord},
\]
then $R\in\BTr(H)$.
\end{proposition}

\begin{proof}
The inclusion $\BT(H)_+\subseteq(\BTr(H))_+$ follows immediately from the
definition of $\BTr(H)$.  Conversely, every positive element of $\BTr(H)$
belongs to $\BT(H)$ because $\BTr(H)\subseteq\BT(H)$; hence
\eqref{eq:BTr-positive-cone} holds.

We first prove the ideal property in its positive form.  Let
\[
 0\le R\le S,
 \qquad R\in\mathcal L_b^A(H),\quad S\in\BTr(H)_+=\BT(H)_+.
\]
Let $k\in A_+$ be an admissible $T$-strong bound for $S$.  Since $R$ is
positive, for every $x\in H$,
\[
 |Rx|\le R|x|\le S|x|.
\]
By monotonicity of the conditional norm,
\[
 \norm{Rx}
 \le \norm{S|x|}
 \le k\norm{|x|}
 =k\norm{x}.
\]
Thus $R\in\BT(H)_+$, and therefore $R\in\BTr(H)$.

Now suppose $|R|_{\rm ord}\le |S|_{\rm ord}$ with $S\in\BTr(H)$.  By
Theorem~\ref{thm:T-regular-lattice}, $|S|_{\rm ord}\in\BT(H)_+$.  The
positive case just proved gives $|R|_{\rm ord}\in\BT(H)_+$.  Since
$R^+,R^-\le |R|_{\rm ord}$ in $\mathcal L_b^A(H)$, the same domination
argument yields $R^+,R^-\in\BT(H)_+$, and hence
$R=R^+-R^-\in\BTr(H)$.
\end{proof}

\begin{theorem}
\label{thm:BTr-Dedekind-complete}
The vector lattice $\BTr(H)$ is Dedekind complete.  More precisely, if
$(S_\alpha)$ is an increasing net in $\BTr(H)$ which is bounded above in
$\BTr(H)$ and
\[
 S=\sup_\alpha S_\alpha,
\]
then for every $x\in H_+$,
\begin{equation}
\label{eq:BTr-pointwise-sup}
 \boxed{Sx=\sup_\alpha S_\alpha x.}
\end{equation}
The analogous statement holds for decreasing nets and infima.  For an
arbitrary nonempty family bounded above, its supremum is obtained from the
directed net of its finite suprema.
\end{theorem}

\begin{proof}
By \cite[Theorem~5.1]{ChamberlainWortel},
$\mathcal L_b^A(H)$ is a Dedekind-complete vector lattice.  Let
$\varnothing\ne\mathcal S\subseteq\BTr(H)$ be bounded above in $\BTr(H)$,
and choose an upper bound $U\in\BTr(H)$.  Its supremum
\[
 V:=\sup_{\mathcal L_b^A(H)}\mathcal S
\]
exists in $\mathcal L_b^A(H)$.  Fix $S_0\in\mathcal S$.  Then
\[
 0\le V-S_0\le U-S_0.
\]
Since $U-S_0\in\BTr(H)$ and $\BTr(H)$ is an order ideal by
Proposition~\ref{prop:BTr-order-ideal}, it follows that
$V-S_0\in\BTr(H)$, hence $V\in\BTr(H)$.  Thus the ambient supremum belongs
to $\BTr(H)$ and is also the supremum there.  This proves Dedekind
completeness.

If $(S_\alpha)$ is increasing, \cite[Theorem~5.1]{ChamberlainWortel}
also states that monotone suprema in $\mathcal L_b^A(H)$ are computed pointwise
on the positive cone.  Since the supremum in $\BTr(H)$ coincides with the
ambient one, \eqref{eq:BTr-pointwise-sup} follows.  The assertion for
infima follows by applying this result to $(-S_\alpha)$, and the final
statement follows from the standard directed net of finite suprema.
\end{proof}

\begin{proposition}
\label{prop:operator-norm-monotone}
If $S,R\in\BT(H)_+$ and $0\le S\le R$, then
\begin{equation}
\label{eq:operator-norm-monotone}
 \boxed{\opnorm{S}\le\opnorm{R}.}
\end{equation}
\end{proposition}

\begin{proof}
Let $k\in A_+$ be any admissible $T$-strong bound for $R$.  For every
$x\in H$,
\[
 |Sx|\le S|x|\le R|x|,
\]
and therefore
\[
 \norm{Sx}\le\norm{R|x|}\le k\norm{x}.
\]
Thus every admissible bound for $R$ is also admissible for $S$.  Taking the
infimum of admissible bounds gives \eqref{eq:operator-norm-monotone}.
\end{proof}

\begin{remark}
The monotonicity in Proposition~\ref{prop:operator-norm-monotone} does not
make $\opnorm{\cdot}$ a lattice norm on $\BTr(H)$ in the stronger sense
$\opnorm{|S|_{\rm ord}}=\opnorm{S}$.  The scalar two-dimensional example
in Section~\ref{sec:comparison-moduli} gives an explicit failure of this
identity; see Example~\ref{ex:two-moduli}.
\end{remark}

\begin{remark}
\label{rem:not-all-order-bounded}
We do \emph{not} claim that
\[
 \BTr(H)=\BT(H)\cap L_b(H,H).
\]
Such an equality would require an automatic $T$-strong boundedness theorem
for the positive and negative Riesz--Kantorovich parts of an arbitrary
order-bounded element of $\BT(H)$.  No such theorem is used here.  The
restriction to $\BTr(H)$ is precisely what prevents the circularity that
would arise from applying the adjoint to $S^+$, $S^-$ or
$|S|_{\rm ord}$ before their membership in $\BT(H)$ has been established.
\end{remark}

\subsection{The order modulus and the adjoint}

\begin{theorem}
\label{thm:adjoint-modulus}
The adjoint maps $\BTr(H)$ onto itself and is an order isomorphism of the
vector lattice $\BTr(H)$.  In particular, for every $S\in\BTr(H)$,
\begin{equation}
\label{eq:adjoint-modulus}
 \boxed{\bigl(|S|_{\rm ord}\bigr)^*=|S^*|_{\rm ord}.}
\end{equation}
Moreover,
\begin{equation}
\label{eq:adjoint-positive-negative-parts}
 (S^+)^*=(S^*)^+,
 \qquad
 (S^-)^*=(S^*)^-.
\end{equation}
Since $\BTr(H)$ is Dedekind complete, the adjoint preserves every bounded
supremum and infimum: for every nonempty family $(S_\alpha)$ which is
bounded above or below, respectively,
\begin{equation}
\label{eq:adjoint-arbitrary-sup-inf}
 \left(\sup_\alpha S_\alpha\right)^*
 =\sup_\alpha S_\alpha^*,
 \qquad
 \left(\inf_\alpha S_\alpha\right)^*
 =\inf_\alpha S_\alpha^*.
\end{equation}
In particular, $S_\alpha\uparrow S$ implies
$S_\alpha^*\uparrow S^*$.
\end{theorem}

\begin{proof}
If $S=P-N$ with $P,N\in\BT(H)_+$, then
\[
 S^*=P^*-N^*.
\]
By Theorem~\ref{thm:order-positive-adjoint}, $P^*,N^*\in\BT(H)_+$, so
$S^*\in\BTr(H)$.  Since the adjoint is involutive, it maps
$\BTr(H)$ onto itself.  The same theorem shows that
\[
 R\le Q\quad\Longleftrightarrow\quad R^*\le Q^*,
 \qquad R,Q\in\BTr(H),
\]
so $*$ is an order isomorphism.

Every order isomorphism of vector lattices preserves suprema, infima and
moduli.  Hence
\[
 \bigl(|S|_{\rm ord}\bigr)^*=|S^*|_{\rm ord}.
\]
Applying the same observation to $S\vee0$ and $(-S)\vee0$ gives the
identities for $S^+$ and $S^-$.

Finally, a bijective order isomorphism preserves every supremum and infimum
that exists: if $U=\sup_\alpha S_\alpha$, then $U^*$ is an upper bound of
$(S_\alpha^*)$, and applying the inverse order isomorphism $*$ to any other
upper bound shows that $U^*$ is the least one.  The infimum statement is
analogous.  Dedekind completeness from
Theorem~\ref{thm:BTr-Dedekind-complete} then yields
\eqref{eq:adjoint-arbitrary-sup-inf} and the monotone-convergence
statement.
\end{proof}

\begin{remark}
The argument above deliberately avoids defining
$\bigl(|S|_{\rm ord}\bigr)^*$ until
Theorem~\ref{thm:T-regular-lattice} has first shown that
$|S|_{\rm ord}\in\BT(H)$.  Thus no adjoint is applied to a merely
order-bounded operator outside its established domain.
\end{remark}

% ============================================================
\section{Positive polynomial calculus and square roots}
% ============================================================

We now develop, directly in the conditional $L^2$ setting, the small
piece of functional calculus needed to define $(S^*S)^{1/2}$.  The point
is that no universal completeness of $A=R(T)$ is required.  We use only
that $A$ is a Dedekind complete unital Archimedean $f$-algebra.  Hence it
is uniformly complete.  Multiplication by an element of $A$ is an
orthomorphism (and therefore order continuous), and the usual positive
square roots and bounded inverses are available in $A$; see
\cite[Theorems~3.4 and~3.9]{HuijsmansDePagterIdeal} and, for the
square-root result, also
\cite[Theorem~4.2 and Corollary~4.3]{BeukersHuijsmansDePagter}.  This
scalar square-root theory should be distinguished from the operator square
root proved below.

For a self-adjoint operator $P\in\BT(H)$ we write
\[
  P\ge_h0
  \quad\Longleftrightarrow\quad
  \inner{Px}{x}\ge0\quad(x\in H),
\]
and $P\le_h Q$ means $Q-P\ge_h0$.

\begin{lemma}
\label{lem:positive-form-CS}
Let $B:H\times H\to A$ be symmetric and $A$-bilinear, and suppose
$B(x,x)\ge0$ for every $x\in H$.  Then
\[
 |B(x,y)|^2\le B(x,x)B(y,y),\qquad x,y\in H.
\]
\end{lemma}

\begin{proof}
Put $a=B(x,x)$, $b=B(x,y)$ and $c=B(y,y)$.  For every $\lambda\in A$,
\[
 0\le B(x+\lambda y,x+\lambda y)=a+2\lambda b+\lambda^2c.
\]
Fix $\varepsilon>0$ and put $c_\varepsilon=c+\varepsilon e$.  Since
$c_\varepsilon\ge\varepsilon e$, it is invertible in $A$ and
$c_\varepsilon^{-1}\ge0$.  Taking
$\lambda=-bc_\varepsilon^{-1}$ gives
\[
 0\le a-2b^2c_\varepsilon^{-1}
       +b^2c\,c_\varepsilon^{-2}
 \le a-b^2c_\varepsilon^{-1},
\]
because $0\le c\le c_\varepsilon$.  Hence
\[
 b^2\le a(c+\varepsilon e).
\]
As $\varepsilon\downarrow0$, order continuity of multiplication by
$a\ge0$ yields $a(c+\varepsilon e)\downarrow ac$.  Therefore
$b^2\le ac$.  Since $b^2=|b|^2$ in an $f$-algebra, the result follows.
\end{proof}

\begin{lemma}
\label{lem:positive-contraction-square}
If $P=P^*$ and $0\le_hP\le_h I$, then
\[
 0\le_hP^2\le_hP.
\]
\end{lemma}

\begin{proof}
The form
\[
 B_P(x,y):=\inner{Px}{y}
\]
is symmetric, $A$-bilinear and positive.  Lemma~\ref{lem:positive-form-CS}
therefore gives
\[
 |\inner{Px}{y}|^2
 \le \inner{Px}{x}\inner{Py}{y}.
\]
Taking $y=Px$ and using $P\le_hI$ at the vector $Px$ gives
\[
 \inner{Px}{Px}^{\,2}
 \le \inner{Px}{x}\inner{P(Px)}{Px}
 \le \inner{Px}{x}\inner{Px}{Px}.
\]
Set $a=\inner{Px}{Px}$ and $b=\inner{Px}{x}$.  Thus $a,b\in A_+$ and
$a^2\le ab$.  We claim that $a\le b$.  Indeed, with
$d=(a-b)^+$, multiplication by $a\ge0$ is a lattice homomorphism, so
\[
 ad=a(a-b)^+=(a(a-b))^+=0.
\]
Since $0\le d\le a$, we have $0\le d^2\le ad=0$.  An Archimedean
$f$-algebra is semiprime, hence $d=0$.  Thus $a\le b$.
Consequently
\[
 \inner{P^2x}{x}=\inner{Px}{Px}\le\inner{Px}{x},
\]
which proves $P^2\le_hP$.  Positivity of $P^2=P^*P$ is immediate.
\end{proof}

\begin{lemma}
\label{lem:polynomial-sandwich}
If $Q=Q^*\ge_h0$ and $R=R^*$ commutes with $Q$, then
\[
 QR^2\ge_h0.
\]
\end{lemma}

\begin{proof}
Since $QR^2=RQR$, for every $x\in H$,
\[
 \inner{QR^2x}{x}=\inner{QRx}{Rx}\ge0.
\]
\end{proof}

\begin{theorem}
\label{thm:positive-polynomial-calculus}
Let $P=P^*$ satisfy $0\le_hP\le_hI$.  If $p\in\mathbb R[t]$ and
$p(t)\ge0$ for every $t\in[0,1]$, then
\[
 \boxed{p(P)\ge_h0.}
\]
Consequently, if $p,q\in\mathbb R[t]$ and $p(t)\le q(t)$ on $[0,1]$,
then $p(P)\le_hq(P)$.
\end{theorem}

\begin{proof}
We use the classical Markov--Luk\'acs representation for a polynomial nonnegative
on a compact interval; see, for example, \cite{PowersReznick}.  If $\deg p=2m$, there are real polynomials
$a,b$ such that
\[
 p(t)=a(t)^2+t(1-t)b(t)^2.
\]
Hence
\[
 p(P)=a(P)^2+P(I-P)b(P)^2.
\]
The first term is positive because $a(P)=a(P)^*$ and
$a(P)^2=a(P)^*a(P)$.  By Lemma~\ref{lem:positive-contraction-square},
$P(I-P)=P-P^2\ge_h0$.  It commutes with $b(P)$, so
Lemma~\ref{lem:polynomial-sandwich} makes the second term positive.

If $\deg p=2m+1$, Markov--Luk\'acs gives
\[
 p(t)=t a(t)^2+(1-t)b(t)^2.
\]
Thus
\[
 p(P)=Pa(P)^2+(I-P)b(P)^2,
\]
and both terms are positive by Lemma~\ref{lem:polynomial-sandwich}.
The comparison assertion follows by applying the first part to $q-p$.
\end{proof}

\begin{corollary}
\label{cor:uniform-polynomial-control}
Let $P=P^*$ and $0\le_hP\le_hI$.  If $p\in\mathbb R[t]$ and
\[
 |p(t)|\le\varepsilon\qquad(0\le t\le1),
\]
where $\varepsilon\ge0$ is real, then
\[
 \norm{p(P)x}\le\varepsilon\norm{x},\qquad x\in H.
\]
\end{corollary}

\begin{proof}
The polynomial $\varepsilon^2-p(t)^2$ is nonnegative on $[0,1]$.
Theorem~\ref{thm:positive-polynomial-calculus} therefore gives
\[
 p(P)^2\le_h\varepsilon^2I.
\]
Since $p(P)$ is self-adjoint,
\[
 \norm{p(P)x}^2
 =\inner{p(P)^2x}{x}
 \le\varepsilon^2\norm{x}^2.
\]
Taking the positive square root in $A$ gives the assertion.
\end{proof}

\begin{theorem}
\label{thm:positive-square-root}
Let $P\in\BT(H)$ be self-adjoint and inner-product positive.  Then there
exists a unique self-adjoint inner-product positive operator
$Q\in\BT(H)$ such that
\[
 \boxed{Q^2=P.}
\]
We denote it by $P^{1/2}$.
\end{theorem}

\begin{proof}
Choose an admissible bound $k\in\mathcal M(P)$ and put
\[
 h=e+k\in A_+.
\]
Then $h\ge e$, so the bounded inversion property gives that $h$ is
invertible with $0<h^{-1}\le e$, and $h^{1/2}$
exists in $A_+$.  Set
\[
 B=h^{-1}P.
\]
Since scalar multiplication by elements of $A$ commutes with every
$A$-linear operator, $B=B^*$ and $B\ge_h0$.  Moreover, for $x\in H$,
conditional Cauchy--Schwarz and $k\in\mathcal M(P)$ give
\[
 \inner{Px}{x}\le\norm{Px}\norm{x}
 \le k\norm{x}^2\le h\norm{x}^2.
\]
Thus $0\le_hB\le_hI$.

For $0\le t\le1$ use the binomial expansion
\[
 \sqrt t
 =1-\sum_{j=1}^{\infty}c_j(1-t)^j,
 \qquad
 c_j=-(-1)^j\binom{1/2}{j}>0.
\]
Since $\sum_{j\ge1}c_j=1$, the series converges uniformly on $[0,1]$.
Let
\[
 r_n(t)=1-\sum_{j=1}^{n}c_j(1-t)^j.
\]
Then $r_n\to\sqrt t$ uniformly on $[0,1]$, and
$0\le r_n(t)\le1$ there.  Put $R_n=r_n(B)$.

If
\[
 \varepsilon_{m,n}
 :=\sup_{0\le t\le1}|r_n(t)-r_m(t)|,
\]
Corollary~\ref{cor:uniform-polynomial-control} yields
\[
 \norm{(R_n-R_m)x}\le\varepsilon_{m,n}\norm{x}.
\]
For every fixed $x$, $(R_nx)$ is therefore $T$-strong Cauchy.  Strong
completeness of $H$ \cite{KalauchKuoWatsonStrong} gives a vector $Rx$
with $R_nx\to Rx$ $T$-strongly.  We spell out the two points needed here.
For $x,y\in H$, continuity of addition for the $A$-valued norm gives
\[
 R(x+y)=Rx+Ry.
\]
If $a\in A$, then
\[
 \norm{R_n(ax)-aRx}=|a|\,\norm{R_nx-Rx}\longrightarrow_o0,
\]
because multiplication by $|a|$ is order continuous on $A$; uniqueness
of the strong limit therefore gives $R(ax)=aRx$.  Thus $R$ is $A$-linear.
Moreover $0\le r_n\le1$, and
Corollary~\ref{cor:uniform-polynomial-control} gives
$\norm{R_nx}\le\norm{x}$.  Since
\[
 \norm{Rx}\le\norm{Rx-R_nx}+\norm{R_nx},
\]
passing to the order limit yields
\[
 \norm{Rx}\le\norm{x}.
\]
Hence $R\in\BT(H)$, with the admissible bound $e$.

Again $0\le r_n\le1$, so
Theorem~\ref{thm:positive-polynomial-calculus} gives $R_n\ge_h0$ and
$I-R_n\ge_h0$.  For fixed $x$,
\[
 |\inner{(R_n-R)x}{x}|
 \le \norm{(R_n-R)x}\norm{x}\longrightarrow_o0,
\]
so the positive cone of $A$ being order closed implies
$0\le_hR\le_hI$.  Finally, for $x,y\in H$,
\[
\begin{aligned}
 |\inner{Rx}{y}-\inner{x}{Ry}|
 &\le \norm{(R-R_n)x}\norm{y}
      +\norm{x}\norm{(R-R_n)y}\longrightarrow_o0.
\end{aligned}
\]
Since every $R_n$ is self-adjoint, it follows that $R=R^*$.

The polynomials $r_n(t)^2$ converge uniformly to $t$ on $[0,1]$.
Corollary~\ref{cor:uniform-polynomial-control} therefore implies
\[
 \norm{(R_n^2-B)x}\longrightarrow0
\]
$T$-strongly, uniformly relative to $\norm{x}$.  On the other hand
$0\le r_n\le1$ implies $\norm{R_nz}\le\norm z$ by the same corollary
applied to $r_n$, and hence
\[
 \norm{R_n^2x-R^2x}
 \le \norm{R_n(R_n-R)x}+\norm{(R_n-R)Rx}
 \longrightarrow0.
\]
Uniqueness of strong limits gives $R^2=B$.

Define
\[
 Q=h^{1/2}R.
\]
Then $Q=Q^*$, $Q\ge_h0$, and, since scalars commute with $R$,
\[
 Q^2=hR^2=hB=P.
\]
This proves existence.

For uniqueness, let $Q_1=Q_1^*\ge_h0$ and $Q_1^2=P$.  Choose an
admissible bound $\ell\in\mathcal M(Q_1)$ and put $g=e+\ell$.  Then
$C=g^{-1}Q_1$ is a positive contraction: indeed
\[
 \inner{Cx}{x}=g^{-1}\inner{Q_1x}{x}
 \le g^{-1}\norm{Q_1x}\norm{x}
 \le g^{-1}\ell\norm{x}^2\le\norm{x}^2.
\]
Moreover $C^2=g^{-2}P$.  Since $r_n(t^2)\to t$ uniformly on $[0,1]$,
Corollary~\ref{cor:uniform-polynomial-control}, applied to the positive
contraction $C$, gives
\[
 r_n(C^2)=r_n(g^{-2}P)\longrightarrow C
\]
strongly.  Multiplying the last convergence by the fixed scalar $g$ gives
\[
 g\,r_n(g^{-2}P)\longrightarrow Q_1
\]
strongly.  Hence every positive square root $Q_1$ belongs to the strong
closure of the polynomial algebra generated by $P$ and the scalar
multiplication operators.

Let $Q_1,Q_2$ be two positive square roots of $P$.  The preceding
observation implies that $Q_1$ and $Q_2$ commute, but the passage to the
strong limit uses boundedness and is worth recording.  For example, let
$A_n$ be polynomial--scalar approximants with $A_nx\to Q_1x$ strongly for
every $x$.  Since $Q_2$ commutes with $P$ and with scalar multiplication,
$Q_2A_n=A_nQ_2$.  Because $Q_2\in\BT(H)$,
\[
 Q_2A_nx\longrightarrow Q_2Q_1x,
\]
while, by the defining strong convergence of $A_n$,
\[
 A_nQ_2x\longrightarrow Q_1Q_2x.
\]
Uniqueness of strong limits gives $Q_2Q_1=Q_1Q_2$.

We shall use one consequence of the existence part already proved.  If
$C,D\ge_h0$ are commuting self-adjoint operators, then $CD\ge_h0$.
Indeed, let $F=C^{1/2}$ be the square root furnished by the construction
above.  Since $F$ is a strong limit of polynomials in $C$ with coefficients
coming from scalar multiplication by elements of $A=R(T)$, every
$T$-strongly bounded operator commuting with $C$ commutes with those
approximants and hence, by the same bounded-limit argument as above, with
$F$.  In particular $D$ commutes with $F$, and therefore
\[
 CD=FDF\ge_h0.
\]

Set
\[
 X=(Q_1-Q_2)^2,\qquad Y=(Q_1+Q_2)^2.
\]
Then $X,Y\ge_h0$, they commute, and
\[
 Y-X=4Q_1Q_2\ge_h0.
\]
Thus $0\le_hX\le_hY$.  Moreover,
\[
 XY=(Q_1-Q_2)^2(Q_1+Q_2)^2
    =(Q_1^2-Q_2^2)^2=0.
\]
Since $X$ and $Y-X$ are commuting positive operators, their product is
positive.  Therefore
\[
 0\le_hX^2\le_hXY=0,
\]
so $X^2=0$.  Indeed, for every $x$,
\[
 \norm{Xx}^2=\inner{X^2x}{x}=0,
\]
and hence $X=0$.  Finally,
\[
 \norm{(Q_1-Q_2)x}^2
 =\inner{Xx}{x}=0
\]
for every $x$, whence $Q_1=Q_2$.  The positive square root is unique.
\end{proof}

\begin{definition}
For $S\in\BT(H)$ define
\[
 \boxed{|S|_*:=(S^*S)^{1/2}.}
\]
This is well defined by Theorem~\ref{thm:positive-square-root}, since
$S^*S$ is self-adjoint and
\[
 \inner{S^*Sx}{x}=\norm{Sx}^2\ge0.
\]
\end{definition}

\begin{remark}
Wright's spectral theorem for normal operators on Kaplansky--Hilbert
modules \cite{WrightSpectral} provides an important classical analogue of
this construction.  We do not invoke that theorem here: its coefficient
algebra is a Stone algebra, whereas the present argument uses only the
Dedekind complete $f$-algebra structure of $R(T)$ and the strong
completeness of $L^2(T)$.  The square root above is obtained directly by
positive polynomial calculus and binomial approximation.
\end{remark}

% ============================================================
\section{Comparison of the order and spectral moduli}
\label{sec:comparison-moduli}
% ============================================================

We now compare the two modulus constructions available in the present
setting.  To avoid ambiguity, throughout this section we write
\[
 |S|_{\rm ord}
\]
for the Riesz--Kantorovich modulus of a $T$-regular operator and
\[
 |S|_*=(S^*S)^{1/2}
\]
for the spectral modulus introduced above.  The two constructions have
quite different origins: the first belongs to the vector-lattice structure
of $\BTr(H)$, whereas the second is determined by the adjoint and
the $R(T)$-valued inner product.

\subsection{Multiplication operators}

For $a\in A=R(T)$ let $M_a:H\to H$ be the multiplication operator
\[
 M_ax=ax.
\]
Since $H$ is an $A$-module, $M_a$ is well defined.  Moreover,
\[
 \norm{M_ax}^2
 =T(a^2x^2)=a^2T(x^2)=|a|^2\norm{x}^2,
\]
and hence
\[
 \norm{M_ax}=|a|\norm{x}.
\]
Thus $M_a\in\BT(H)$.  Since
\[
 M_a=M_{a^+}-M_{a^-}
\]
and $M_{a^+},M_{a^-}$ are order positive, in fact
$M_a\in\BTr(H)$.  Also
\[
 \inner{M_ax}{y}=T(axy)=T(xay)=\inner{x}{M_ay},
\]
so that $M_a^*=M_a$.

\begin{proposition}
\label{prop:multiplication-moduli}
For every $a\in R(T)$,
\[
 \boxed{|M_a|_{\rm ord}=|M_a|_*=M_{|a|}.}
\]
\end{proposition}

\begin{proof}
The lattice modulus of a multiplication operator is the multiplication
operator by the lattice modulus of its coefficient.  Indeed, multiplication
by $a$ is an orthomorphism and the canonical embedding of the unital
$f$-algebra $A$ into its orthomorphism algebra is a lattice homomorphism;
therefore
\[
 |M_a|_{\rm ord}=M_{|a|}.
\]
On the other hand,
\[
 M_a^*M_a=M_{a^2}=M_{|a|}^2.
\]
For every $x\in H$,
\[
 \inner{M_{|a|}x}{x}=T(|a|x^2)=|a|T(x^2)\ge0,
\]
so $M_{|a|}\ge_h0$.  By uniqueness in
Theorem~\ref{thm:positive-square-root},
\[
 |M_a|_*=(M_a^*M_a)^{1/2}=M_{|a|}.
\]
\end{proof}

\subsection{The two moduli need not coincide}

The failure already occurs in the classical two-dimensional case and is
therefore not a peculiarity of conditional $L^2$-spaces.

\begin{example}
\label{ex:two-moduli}
Let $\Omega=\{1,2\}$ with equal weights and let $T$ be expectation onto the
constants.  Then $A=R(T)=\mathbb R e$ and $H=L^2(T)=\mathbb R^2$ with the
coordinatewise lattice order and
\[
 \inner{x}{y}=\frac12(x_1y_1+x_2y_2)e.
\]
Consider
\[
 S=\begin{pmatrix}1&1\\[2pt]1&-1\end{pmatrix}.
\]
In the present scalar finite-dimensional setting, every linear operator is
regular and $T$-strongly bounded.  Hence $S\in\BTr(H)$.  Moreover, $S^*=S$ and
$S^*S=2I$, whence
\[
 |S|_*=\sqrt2\,I.
\]
For $x=(x_1,x_2)\ge0$, the Riesz--Kantorovich formula gives
\[
 |S|_{\rm ord}x
 =\sup_{|u|\le x}|Su|
 =(x_1+x_2,x_1+x_2).
\]
Consequently
\[
 |S|_{\rm ord}
 =\begin{pmatrix}1&1\\[2pt]1&1\end{pmatrix}
 \neq \sqrt2 I=|S|_*.
\]
In fact neither modulus dominates the other in the lattice order.  Applied
to $e_1=(1,0)$, their two differences are respectively
\[
 (|S|_{\rm ord}-|S|_*)e_1=(1-\sqrt2,1),
 \qquad
 (|S|_*-|S|_{\rm ord})e_1=(\sqrt2-1,-1),
\]
and neither vector is positive.  The same example also shows that the
$A$-valued operator norm is not a lattice norm on $\BTr(H)$: here it is the
usual Euclidean operator norm multiplied by $e$, and therefore
\[
 \opnorm{S}=\sqrt2\,e,
 \qquad
 \opnorm{|S|_{\rm ord}}=2e.
\]
\end{example}

The preceding example gives the algebraic obstruction directly:
\[
 |S|_{\rm ord}^{\,2}
 =\begin{pmatrix}2&2\\[2pt]2&2\end{pmatrix}
 \neq 2I=S^*S.
\]
The next result identifies the obstruction without any self-adjointness
assumption.

\subsection{The general cross-term identity}

The preceding example gives the algebraic obstruction directly.  The
self-adjointness assumption is not needed.  Comparing the order modulus
through $|S|_{\rm ord}^{*}|S|_{\rm ord}$ gives an identity for arbitrary
$T$-regular operators.

\begin{theorem}
\label{thm:general-cross-term-identity}
Let $S\in\BTr(H)$.  Then $S^*\in\BTr(H)$ and
$|S|_{\rm ord}\in\BTr(H)_+=\BT(H)_+$, and
\[
 \boxed{
 |S|_{\rm ord}^{*}|S|_{\rm ord}-S^*S
 =2\bigl((S^*)^+S^-+(S^*)^-S^+\bigr).
 }
 \tag{\ensuremath{\mathrm{CF}}}
\]
In particular,
\[
 \boxed{\qquad S^*S\le_o |S|_{\rm ord}^{*}|S|_{\rm ord}.\qquad}
\]
Moreover, equality holds if and only if
\[
 (S^*)^+S^-=0
 \qquad\text{and}\qquad
 (S^*)^-S^+=0.
\]
\end{theorem}

\begin{proof}
By Theorem~\ref{thm:T-regular-lattice},
\[
 S=S^+-S^-,\qquad |S|_{\rm ord}=S^++S^-,
\]
with $S^+,S^-,|S|_{\rm ord}\in\BT(H)$.  By
Theorem~\ref{thm:adjoint-modulus}, $S^*\in\BTr(H)$ and
\[
 (|S|_{\rm ord})^*=|S^*|_{\rm ord},\qquad
 (S^+)^*=(S^*)^+,
 \qquad
 (S^-)^*=(S^*)^-.
\]
Put
\[
 P=S^+,\quad N=S^-,\quad
 \widehat P=(S^*)^+,\quad \widehat N=(S^*)^-.
\]
Then
\[
 |S|_{\rm ord}=P+N,
 \qquad
 |S|_{\rm ord}^*=\widehat P+\widehat N,
 \qquad
 S^*=\widehat P-\widehat N.
\]
Therefore
\begin{align*}
 |S|_{\rm ord}^*|S|_{\rm ord}
 &= (\widehat P+\widehat N)(P+N)\\
 &=\widehat PP+\widehat PN+\widehat NP+\widehat NN,
\end{align*}
whereas
\begin{align*}
 S^*S
 &= (\widehat P-\widehat N)(P-N)\\
 &=\widehat PP-\widehat PN-\widehat NP+\widehat NN.
\end{align*}
Subtracting yields (CF).

All four operators $P,N,\widehat P,\widehat N$ are order positive, and the
composition of positive operators is positive.  Hence
\[
 (S^*)^+S^-\ge_o0,
 \qquad
 (S^*)^-S^+\ge_o0.
\]
This proves the operator inequality.  Since a sum of two positive operators
is zero if and only if both summands are zero, the equality characterization
follows as well.
\end{proof}

\begin{remark}
The hypothesis $S\in\BTr(H)$ is substantive.  It ensures simultaneously
that the Riesz--Kantorovich parts are defined and that they belong to
$\BT(H)$, so that all adjoints in (CF) are legitimate.  For a merely
order-bounded $S\in\BT(H)$ this additional $T$-strong boundedness of
$S^+$, $S^-$ and $|S|_{\rm ord}$ has not been assumed or derived; see
Remark~\ref{rem:not-all-order-bounded}.
\end{remark}

\begin{corollary}
\label{thm:cross-term-identity}
Let $S\in\BTr(H)$ be self-adjoint.  Then $S^+$ and $S^-$
are self-adjoint and
\[
 \boxed{
 |S|_{\rm ord}^{\,2}-S^*S
 =2\bigl(S^+S^-+S^-S^+\bigr)\ge_o0.
 }
 \tag{\ensuremath{\dagger}}
\]
Moreover,
\[
 |S|_{\rm ord}^{\,2}=S^*S
 \quad\Longleftrightarrow\quad
 S^+S^-=0
 \quad\Longleftrightarrow\quad
 S^-S^+=0.
\]
\end{corollary}

\begin{proof}
If $S=S^*$, then $(S^*)^+=S^+$, $(S^*)^-=S^-$ and
$|S|_{\rm ord}^*=|S|_{\rm ord}$.  The identity follows immediately from
Theorem~\ref{thm:general-cross-term-identity}.  Both cross products are
positive operators, so their sum vanishes exactly when both vanish.  Since
$S=S^*$ and the adjoint preserves lattice operations on $\BTr(H)$, both
$S^+$ and $S^-$ are self-adjoint.  Hence
\[
 (S^+S^-)^*=(S^-)^*(S^+)^*=S^-S^+.
\]
Therefore $S^+S^-=0$ if and only if $S^-S^+=0$.
\end{proof}

\begin{theorem}
\label{thm:moduli-equality-criterion}
Let $S\in\BTr(H)$ be self-adjoint, and suppose in addition
that
\[
 |S|_{\rm ord}\ge_h0.
\]
Then
\[
 \boxed{
 |S|_{\rm ord}=|S|_*
 \quad\Longleftrightarrow\quad
 S^+S^-=0.
 }
\]
Equivalently, equality holds if and only if
$S^+S^-+S^-S^+=0$.
\end{theorem}

\begin{proof}
If $|S|_{\rm ord}=|S|_*$, then
\[
 |S|_{\rm ord}^{\,2}=|S|_*^2=S^*S,
\]
and Corollary~\ref{thm:cross-term-identity} gives $S^+S^-=0$.
Conversely, if $S^+S^-=0$, the same corollary gives
\[
 |S|_{\rm ord}^{\,2}=S^*S.
\]
By hypothesis $|S|_{\rm ord}$ is self-adjoint and inner-product positive.
It is therefore the distinguished inner-product-positive square root of
$S^*S$.  Uniqueness in Theorem~\ref{thm:positive-square-root} yields
\[
 |S|_{\rm ord}=(S^*S)^{1/2}=|S|_*.
\]
\end{proof}

\begin{remark}
The hypothesis $|S|_{\rm ord}\ge_h0$ is needed for the converse argument
in the preceding theorem: once $|S|_{\rm ord}^{\,2}=S^*S$ is known, it is
this inner-product positivity that identifies $|S|_{\rm ord}$ with the
distinguished positive square root.  The general identity (CF) and the order
inequality $S^*S\le_o|S|_{\rm ord}^*|S|_{\rm ord}$ require no
self-adjointness and no inner-product positivity of $|S|_{\rm ord}$.
\end{remark}

\begin{remark}
For multiplication operators,
\[
 (M_a)^+=M_{a^+},\qquad (M_a)^-=M_{a^-},
\]
and $a^+a^-=0$, so the cross products vanish.  By contrast,
Example~\ref{ex:two-moduli} mixes the two coordinate bands and its cross
terms are non-zero.  The general formula (CF) shows that, without
self-adjointness, the relevant obstruction is not merely the interaction of
$S^+$ and $S^-$, but the interaction of the positive and negative parts of
$S$ with the opposite parts of its adjoint.
\end{remark}

% ============================================================
% Bibliography
% ============================================================


\begin{thebibliography}{99}

\bibitem{AbramovichAliprantis}
Y.~A.~Abramovich and C.~D.~Aliprantis,
\newblock \emph{An Invitation to Operator Theory},
\newblock Graduate Studies in Mathematics, Vol.~50,
American Mathematical Society, Providence, RI, 2002.

\bibitem{AbramovichAliprantisPolyrakis}
Y.~A.~Abramovich, C.~D.~Aliprantis, and I.~A.~Polyrakis,
\newblock Lattice-subspaces and positive projections,
\newblock {\em Proc. Roy. Irish Acad. Sect. A} {\bf 94A} (1994), no.~2,
237--253.

\bibitem{ABPositiveOperators}
C.~D.~Aliprantis and O.~Burkinshaw,
\newblock \emph{Positive Operators},
\newblock Springer, Dordrecht, 2006.

\bibitem{AzouziKuoRamdaneWatson}
Y.~Azouzi, W.-C.~Kuo, K.~Ramdane and B.~A.~Watson,
\newblock Convergence in Riesz spaces with conditional expectation operators,
\newblock \emph{Positivity} \textbf{19} (2015), 647--657.
\newblock \url{https://doi.org/10.1007/s11117-014-0320-6}.

\bibitem{AzouziTrabelsi}
Y.~Azouzi and M.~Trabelsi,
\newblock $L^p$-spaces with respect to conditional expectation on Riesz spaces,
\newblock \emph{Journal of Mathematical Analysis and Applications}
\textbf{447} (2017), 798--816.

\bibitem{BeukersHuijsmansDePagter}
F.~Beukers, C.~B.~Huijsmans and B.~de Pagter,
\newblock Unital embedding and complexification of $f$-algebras,
\newblock \emph{Mathematische Zeitschrift} \textbf{183} (1983), 131--144.
\newblock \url{https://doi.org/10.1007/BF01187219}.

\bibitem{ChamberlainWortel}
T.~Chamberlain and M.~Wortel,
\newblock The Riesz--Kantorovich formulas for $\mathbb L$-vector lattices,
\newblock \emph{Positivity} \textbf{30} (2026), Article~28.
\newblock \url{https://doi.org/10.1007/s11117-026-01184-w}.

\bibitem{ConwayFA}
J.~B.~Conway,
\newblock \emph{A Course in Functional Analysis}, 2nd ed.,
\newblock Graduate Texts in Mathematics 96, Springer, New York, 1990.

\bibitem{HuijsmansDePagterIdeal}
C.~B.~Huijsmans and B.~de Pagter,
\newblock Ideal theory in $f$-algebras,
\newblock \emph{Transactions of the American Mathematical Society}
\textbf{269} (1982), 225--245.
\newblock \url{https://doi.org/10.2307/1998601}.

\bibitem{KalauchKuoWatsonRF}
A.~Kalauch, W.-C.~Kuo and B.~A.~Watson,
\newblock A Hahn--Jordan decomposition and Riesz--Fr\'echet
representation theorem in Riesz spaces,
\newblock \emph{Quaestiones Mathematicae} \textbf{47} (2024),
Suppl.~1, S233--S246.
\newblock
\url{https://doi.org/10.2989/16073606.2023.2287843}.

\bibitem{KalauchKuoWatsonStrong}
A.~Kalauch, W.-C.~Kuo and B.~A.~Watson,
\newblock Strong completeness of a class of $L^2(T)$-type Riesz spaces,
\newblock \emph{Proceedings of the American Mathematical Society,
Series B} \textbf{11} (2024), 243--253.
\newblock
\url{https://doi.org/10.1090/bproc/230}.

\bibitem{KaplanskyHilbert}
I.~Kaplansky,
\newblock Modules over operator algebras,
\newblock \emph{American Journal of Mathematics} \textbf{75} (1953), 839--858.

\bibitem{KuoLabuschagneWatson}
W.-C.~Kuo, C.~C.~A.~Labuschagne and B.~A.~Watson,
\newblock Conditional expectations on Riesz spaces,
\newblock \emph{Journal of Mathematical Analysis and Applications}
\textbf{303} (2005), 509--521.
\newblock
\url{https://doi.org/10.1016/j.jmaa.2004.08.050}.

\bibitem{KuoRoddaWatson}
W.-C.~Kuo, D.~Rodda and B.~A.~Watson,
\newblock Strong sequential completeness of the natural domain of a
conditional expectation operator in Riesz spaces,
\newblock \emph{Proceedings of the American Mathematical Society}
\textbf{147} (2019), 1597--1603.
\newblock \url{https://doi.org/10.1090/proc/14341}.

\bibitem{LabuschagneWatson}
C.~C.~A.~Labuschagne and B.~A.~Watson,
\newblock Discrete stochastic integration in Riesz spaces,
\newblock \emph{Positivity} \textbf{14} (2010), 859--875.
\newblock \url{https://doi.org/10.1007/s11117-010-0089-1}.

\bibitem{Lance}
E.~C.~Lance,
\newblock \emph{Hilbert $C^*$-Modules: A Toolkit for Operator Algebraists},
\newblock London Mathematical Society Lecture Note Series 210,
Cambridge University Press, Cambridge, 1995.
\newblock \url{https://doi.org/10.1017/CBO9780511526206}.

\bibitem{PowersReznick}
V.~Powers and B.~Reznick,
\newblock Polynomials that are positive on an interval,
\newblock \emph{Transactions of the American Mathematical Society}
\textbf{352} (2000), 4677--4692.
\newblock \url{https://doi.org/10.1090/S0002-9947-00-02595-2}.

\bibitem{ReedSimonI}
M.~Reed and B.~Simon,
\newblock \emph{Methods of Modern Mathematical Physics, Vol.~I: Functional Analysis},
\newblock Academic Press, New York, 1980.

\bibitem{WrightSpectral}
J.~D.~M.~Wright,
\newblock A spectral theorem for normal operators on a Kaplansky--Hilbert module,
\newblock \emph{Proceedings of the London Mathematical Society}
(3) \textbf{19} (1969), 258--268.
\newblock \url{https://doi.org/10.1112/plms/s3-19.2.258}.

\bibitem{ZaanenRieszII}
A.~C.~Zaanen,
\newblock \emph{Riesz Spaces II},
\newblock North-Holland Mathematical Library 30, North-Holland, Amsterdam, 1983.

\end{thebibliography}
\end{document}